\documentclass[11pt]{article}
\usepackage{amsmath, amsthm, amssymb}
\usepackage[final]{graphicx}
\usepackage{color}
\usepackage[bookmarks,colorlinks,linkcolor=blue,
	    citecolor=blue,urlcolor=blue]{hyperref}

\title{\Large Selections Without Adjacency on a Rectangular Grid}
\author{Jacob A. Siehler}
\date{}

\newtheorem{prop}{Proposition}[section]
\newtheorem{cor}[prop]{Corollary}
\newtheorem{lem}[prop]{Lemma}

\theoremstyle{definition}

\theoremstyle{remark}
\newtheorem*{rem}{Remark}

\def\hgf{{}_2F_1}
\def\th{^\text{th}}
\def\floor#1{\left\lfloor #1 \right\rfloor}
\def\ceil#1{\left\lceil #1 \right\rceil}

\begin{document}
\maketitle

\begin{abstract}
Using $T(m,n;k)$ to denote the number of ways to make a selection
of $k$ squares from an $m\times n$ rectangular grid with no two
squares in the selection adjacent, we give a formula for $T(2,n;k)$,
prove some identities satisfied by these numbers, and show that
$T(2,n;k)$ is given by a degree $k$ polynomial in $n$.  We give
simple formulas for the first few (most significant) coefficients
of the polynomials.  We give corresponding results for $T(3,n;k)$
as well.  Finally we prove a unimodality theorem which shows, in
particular, how to choose $k$ in order to maximize $T(2,n;k)$.
\end{abstract}

\section{Introduction and main results}

Throughout this paper we will use $T(m,n;k)$ to denote the number
of ways to select $k$ squares from an $m\times n$ grid, with no two
squares in the selection horizontally or vertically adjacent.  For
example, Figure \ref{f1} shows an adjacency-free selection of 6
squares from a $3\times5$ grid.  It turns out that there are 53
ways to make such a selection, so we write $T(3,5;6)=53$.

\begin{figure}[!hbt]
\begin{center}
    \includegraphics{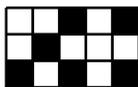}
    \caption{An adjacency-free selection of 6 squares}\label{f1}
\end{center}
\end{figure}

If we fix a value of $m$, then we can tabulate values of $T(m,n;k)$
in a Pascal-style triangular array.  This paper is concerned with
documenting properties of this table for the small special cases
$m=2$ and $m=3$.  Tables \ref{t2ntable} and \ref{t3ntable} below
show the first few rows in these two cases.

This problem is similar in flavor to the ``Problem of the Kings''
discussed in \cite{W} and subsequently \cite{L}, which study
essentially the same problem, only with diagonal adjacencies also
prohibited.  Those papers focus specifically on placing $k=mn$
nonattacking kings on an $(2m)\times(2n)$ chessboard.

\subsection{Results for the \texorpdfstring{$2\times N$}{2 by N}  case}

Table \ref{t2ntable} shows the first few values of $T(2,n;k)$.  Row
sums of this table (and the $T(m,n;k)$ table in general) are dealt
with in \cite{CW}.  The table, read row-by-row, occurs as Sloane's
sequence
\href{http://www.research.att.com/~njas/sequences/A035607}{A035607}
\cite{Sl}, and these numbers occur in the context of counting
integer lattice points of fixed $l_1$ norm in \cite{S}.

\begin{table}[!hbt]
\begin{center}
\begin{tabular}{l|ccccccc}
$n\backslash k$&0&1&2&3&4&5&6 \\
\hline
0&1 \\
1&1& 2\\
2&1& 4& 2\\
3&1& 6& 8& 2\\
4&1& 8& 18& 12& 2\\
5&1& 10& 32& 38& 16&  2\\
6&1& 12& 50& 88& 66& 20& 2\\
\end{tabular}
\caption{Small values of $T(2,n;k)$}\label{t2ntable}
\end{center}
\end{table}
The $2\times n$ case is not particularly difficult to understand
thoroughly, and the following results give fairly complete information
about the table in this case.  The nonzero area of the table is
obvious, but for completeness
we state:

\begin{prop}\label{boundary2} $T(2,n;k)>0$ if and only if $0\le k\le n$.
\end{prop}
The $2\times n$ case admits a simple, explicit formula:

\begin{prop}\label{formulas2} Values of $T(2,n;k)$ are given by the formula
\[T(2,n;k)=\sum_{r=1}^k 2^r\binom{k-1}{r-1}\binom{n-k+1}{r}\]
Equivalently, $T(2,n;k)$ can be expressed in terms of the hypergeometric
$\hgf$ function as \[T(2,n;k)=2(n-k+1)\ \hgf(1-k,k-n,2,2)\]
\end{prop}

\begin{cor}[Polynomial columns for $2\times n$]\label{poly2} For a
fixed $k$, $T(2,n;k)$ is polynomial in $n$; more explicitly,
\[T(2,n;k)=\frac{2^k}{k!}n^k 
    - \frac{2^k}{(k-2)!}n^{k-1}+O\left(n^{k-2}\right)\]
\end{cor}

\begin{rem} 
The leading $2^k/k!$ should be no surprise; it is merely the statement
that most selections on a $2\times n$ grid are adjacency-free when
$n$ is large.  It is therefore the second coefficient that gives
the primary information about how many selections {\it do} contain
at least one adjacency.
\end{rem}

Tabulation of small values in the $2\times n$ table is facilitated
by the following identities.

\begin{prop}[$2\times n$ identities]\label{identities2}
The numbers $T(2,n;k)$ satisfy the Pascal-style identity
\[T(2,n;k)=T(2,n-2;k-1)+T(2,n-1;k-1)+T(2,n-1;k)\]
and the ``hockeystick''-style identity
\[T(2,n;k)=T(2,n-1;k)+\sum_{r=1}^k 2T(2,n-r-1;k-r)\]

\end{prop}

A basic structural feature of the $T(2,n;k)$ table is that the rows
(i.e., for fixed $n$) are unimodal, with the maximum entry occuring
at $k=\lceil n/2 \rceil$.  This elementary observation seems to
require some effort to prove.

\begin{prop}[Unimodal rows for $2\times n$]\label{unimodal2} For
each $n\ge0$, \[T(2,n;0)<T(2,n;1)<\dots<T(2,n;\lceil n/2\rceil)\] and
\[T(2,n;\lceil n/2\rceil)>T(2,n;\lceil n/2\rceil+1)>\dots> T(2,n;n)\]
\end{prop}

I expect that rows of the table $T(m,n;k)$ for any fixed $m$ share
the unimodal property, but can not prove the general result yet.

\subsection{Results for the \texorpdfstring{$3\times N$}{3 by N} case}

\begin{table}[!hbt]
\begin{center}
\begin{tabular}{l|ccccccccc}
$n\backslash k$ &0&1&2&3&4&5&6&7&8\\
\hline
0 &1\\
1 &1& 3& 1\\
2 &1& 6& 8& 2\\
3 &1& 9& 24& 22& 6& 1\\
4 &1& 12& 49& 84& 61& 18& 2\\
5 &1& 15& 83& 215& 276& 174& 53& 9& 1\\
\end{tabular}
\caption{Small values of $T(3,n;k)$}\label{t3ntable}
\end{center}
\end{table}

The nonzero area of the $3\times n$ table is described by the
following proposition.  This is not difficult, but useful to have
in precise terms when implementing calculations involving these
numbers.

\begin{prop}\label{boundary3} For a given $n$, $T(3,n;k)>0$ if and 
only if $0\le k\le\left\lfloor(3n+1)/2\right\rfloor$.  Likewise,
for a given $k$, $T(3,n;k)>0$ if and only if
$n\ge\left\lfloor(2k+1)/3\right\rfloor$.
\end{prop}

The main result here is the following analog to Proposition
\ref{poly2}. Without an explicit formula for $T(3,n;k)$ the proof
is more difficult in this case.

\begin{prop}[Polynomial columns for $3\times n$]\label{poly3}
For each $k\ge1$ there is a polynomial $p_k$ of degree $k$ such that
\[T(3,n;k)=p_k(n)\quad\text{for all }n\ge k\]
More explicitly, these polynomials have the form
\[p_k(n) = \frac{3^k}{k!} n^k 
    - \frac{13(3^{k-2})}{2(k-2)!} n^{k-1}+O\left(n^{k-2}\right)\]
\end{prop}

\begin{rem} With the previous result established, it is not difficult
to work out the polynomials $p_k$ explicitly for any given $k$.
Table \ref{polytable} gives a full expansion of the first few.
Note that, for $n<k$, the true value of $T(3,n;k)$ will generally
differ from $p_k(n)$. It appears that $p_k(n-1)$ does agree
with $T(3,n-1;k)$, but the proof below does not establish this.
\end{rem}

\begin{table}[!hbt]
\begin{center} \renewcommand{\arraystretch}{1.25}
\begin{tabular}{r|c}
$k$ & $p_k(n)$\\
\hline
1&$3n$\\
2&$\frac{1}{2} \left(9 n^2-13 n+6\right)$\\
3&$\frac{1}{2} \left(9 n^3-39 n^2+64 n-40\right)$\\
4&$\frac{1}{8} \left(27 n^4-234 n^3+829 n^2-1430 n+1008\right)$\\
5&$\frac{1}{40} \left(81 n^5-1170 n^4+7215 n^3-23830 n^2+42144 n-31760\right)$
\end{tabular}
\end{center}
\caption{Polynomials $p_k(n)$ for $T(3,n;k)$, $n\ge k$}\label{polytable}
\end{table}

\section{Details for the \texorpdfstring{$2\times N$}{2 by N} case}
\def\T#1#2{T(2,#1;#2)}

\begin{proof}[Proof of Proposition \ref{boundary2}]
$\T n n =2$ is clear, and that implies $\T n k >0$ for $0\le k\le
n$ as well.  Conversely, by the pigeonhole principle, $\T n k =0$
if $k>n$.
\end{proof}

\begin{proof}[Proof of Proposition \ref{formulas2}]
Consider an $n$-tuple of natural numbers that we will call the 
{\it projection} of a selection on the grid: the $i\th$ component
of the projection vector is simply the number of squares selected
in the $i\th$ column of the grid. A selection of $k$ squares on a
$2\times n$ grid obviously gives a projection with $k$ 1's and
$(n-k)$ 0's.

If we want to lift a projection vector to a corresponding selection
on a $2\times n$ grid, the number of ways to do this depends only
on the number of {\it runs} of 1's in the projection vector.  For
example, the projection $\langle 0,1,0,1,1,1,0,0,1,1,0\rangle$ has
3 runs, and the number of ways to lift this to a selection on the
grid is $2^3$.  If there are $r$ runs of 1's, of course, there are
$2^r$ liftings of the projection to a selection on the grid.

Now, the number of $n$-tuples with $k$ 1's organized into $r$ runs is
\[\binom{k-1}{r-1}\binom{n-k+1}{r},\]
as there are $\binom{k-1}{r-1}$ ways to partition the 1's into runs,
and $\binom{n-k+1}{r}$ ways to insert the runs of 1's into a string
of $(n-k)$ 0's.  The number of runs could conceivably be anything
from 1 to $k$, so we arrive at the formula \[\T n k = \sum_{r=1}^k
2^r \binom{k-1}{r-1}\binom{n-k+1}{r}\] as claimed.

The expression in terms of $\hgf$ follows immediately.  Note that
at least one of the first two parameters of the hypergeometric
function (either $1-k$ or $k-n$) is negative for $0\le k\le n$,
meaning that the sum involved is a finite one.
\end{proof}

\begin{proof}[Proof of Corollary \ref{poly2}]
The summation of Proposition \ref{formulas2} makes it clear that
the highest degree of $n$ occurring is $n^k$.  The only summand
contributing an $n^k$ term is $r=k$, and so the highest-degree
coefficient is $2^k/k!$ as claimed.  

For the $n^{k-1}$ coefficient, both the $r=k-1$ and $r=k$ summands
contribute terms.  The contribution of the $r=k-1$ summand is
\[\frac{2^{k-1}}{(k-2)!}n^{k-1}\]
and the contribution of the $r=k$ summand is
\[\frac{2^k\left(-(k-1)-(k)-\dots-(2k-2)\right)}{k!} n^{k-1}\]
which simplifies to
\[\frac{-3\cdot 2^{k-1}}{(k-2)!}n^{k-1}\]
and so combining the two contributions we see that 
\begin{equation*}
\T n k =\frac{2^k}{k!}n^k - \frac{2^k}{(k-2)!}n^{k-1}+O\left(n^{k-2}\right).
\qedhere
\end{equation*}
\end{proof}

\def\Ta#1#2{T_0(2,#1;#2)}
\def\Tb#1#2{T_1(2,#1;#2)}

We introduce the notation $\Ta n k$ to stand for the number of
adjacency-free selections on a $2\times n$ grid, in which the last
column contains no selected squares.  $\Tb n k$ will stand for the
number of selections in which the {\it bottom} square of the last
column is selected.

\begin{lem}
These functions satisfy the following relations:
\begin{align*}
\T  n k &= \Ta n k + 2\Tb n k \tag{R1} \\
\Ta n k &= \T {n-1} k \tag{R2} \\
\Tb n k &= \Ta {n-1} {k-1} + \Tb {n-1} {k-1} \tag{R3} 
\end{align*}
\end{lem}

\begin{proof} 
These are all self-evident from the nature of the problem.
\end{proof}

\begin{proof}[Proof of Proposition \ref{identities2} ($2\times n$ Identities)]
The Pascal-style identity is quickly derived from the relations above:
\begin{alignat*}{2}
\T n k	& = \Ta n k + 2\Tb n k &&\quad\text{(by R1)}\\
	& = \T {n-1} k + 2\left(\Ta {n-1} {k-1} + \Tb {n-1} {k-1} \right) 
	    &&\quad\text{(by R3)}\\
	& = \T {n-1} k + \T {n-2} {k-1} & \\
	& \qquad + \Ta {n-1} {k-1} + 2\Tb {n-1} {k-1} &&\quad\text{(by R2)}\\
	& = \T {n-2} {k-1} + \T {n-1} {k-1} + \T {n-1} k &&\quad\text{(by R1).}
\end{alignat*}
and so is the hockeystick-style identity:
\begin{align*}
\T n k 	&= \T {n-1} k + 2\Tb n k \\
	&= \T {n-1} k + 2\T {n-2} {k-1} + 2\Tb {n-1} {k-1} \\
	&= \qquad \cdots\quad\text{(expand by R3 repeatedly)} \quad \cdots \\
	&= \T {n-1} k + 2\T {n-2} {k-1} \\
	&\qquad +2\T {n-3} {k-2} + \dots + \T {n-k-1} 0.\qedhere
\end{align*}
\end{proof}

Before attempting the proof of Proposition \ref{unimodal2}, we will
need a few preliminary results, which are purely technical.  First,
we introduce the notation
\[\Delta(k,r) := \binom {k-1}{r-1} \binom k r - \binom {k-2}{r-1} \binom{k+1}r\]
and establish the following antisymmetry property of $\Delta$:

\begin{lem}
\label{lem6} $\Delta(k,k+1-r)=-\Delta(k,r)$ for $1\le r\le k$.	
\end{lem}

\begin{proof}
This is just an exercise in elementary properties of binomial
coefficients.
\begin{align*}
\Delta(k,k+1-r)	&= \binom{k-1}{k-r}\binom{k}{k+1-r}
-\binom{k-2}{k-r}\binom{k+1}{k+1-r}\\
    &= \binom{k-1}{r-1}\binom{k}{r-1}-\binom{k-2}{r-2}\binom{k+1}{r} \\
    &= \binom{k-1}{r-1}\binom{k}{r-1}-
	\left[\binom{k-1}{r-1}-\binom{k-2}{r-1}\right]
	\left[\binom k r + \binom k {r-1}\right] \\
    &= \binom{k-2}{r-1}\binom k r + \binom{k-2}{r-1}\binom k {r-1}
	-\binom{k-1}{r-1}\binom k r \\
    &= \binom{k-2}{r-1}\left[\binom{k+1} r - \binom k {r-1} \right]
	+\binom{k-2}{r-1}\binom k {r-1} 
	-\binom{k-1}{r-1}\binom k r \\
    &= \binom{k-2}{r-1}\binom{k+1}r - \binom{k-1}{r-1}\binom k r \\
    &= -\Delta(k,r).\qedhere
\end{align*}
\end{proof}
\begin{rem}
In particular, antisymmetry forces $\Delta(2m-1,m)$ to be 0 for any $m\ge1$.
\end{rem}

\begin{lem}\label{lem7} If $r\le \floor {k/2}$ then $\Delta(k,r)<0$.	
\end{lem}
\begin{proof}
If $r\le\floor {k/2}$ then $2r < k+1$.  It follows that
\[\frac{(k+1)(k-r)}{(k-r+1)(k-1)} > 1,\]
hence
\begin{align*}
\binom {k-2}{r-1}\binom {k+1} r &= \binom {k-1}{r-1} \binom k r
	\times\frac{(k+1)(k-r)}{(k-r+1)(k-1)}  \\
    &> \binom{k-1}{r-1}\binom k r
\end{align*}
and so $\Delta(k,r) < 0$.
\end{proof}

\begin{cor}\label{cor8} $\T{2k-1}{k} > \T{2k-1}{k-1}$ \end{cor}
\begin{proof}
Using the summation formula from Proposition \ref{formulas2}, we have
\begin{align*}
\T {2k-1} k - \T {2k-1} {k-1} &= \sum_{r=1}^k 2^r \Delta(k,r) \\
&= \sum_{r=1}^{\floor {k/2}} \Delta(k,r)\left(2^r-2^{k-r+1}\right) 
    \quad\text {(by Lemma \ref{lem6})}
\end{align*}
and this is positive since both terms in the sum are negative when
$r\le \floor {k/2}$.  (If $k$ is odd, there is an ``unpaired'' term
in the middle of the summation -- but that term has the form
$2^m\Delta(2m-1,m)$, which we have seen to be zero.)
\end{proof}

The point of all the preceding is simply to establish that the
alleged maximum in the $n\th$ row of the $\T n k$ table is larger
than the next entry to the left, when $n$ is odd.  The following
two rather easier results establish a similar fact for rows in which
$n$ is even.

\begin{lem}
$\displaystyle{\binom{k-1}{r-1}\binom{k+1} r - \binom k {r-1} \binom k r > 0}$ 
whenever $1\le r\le k$.
\end{lem}

\begin{proof}
We can expand the first product as 
\[  \binom{k-1}{r-1}\binom{k+1}{r} = 
    \binom{k}{r-1}\binom{k}{r}-
    \binom{k}{r-1}\binom{k-1}{r}+
    \binom{k-1}{r-1}\binom{k}{r}\]
Now, focusing on the last two terms, we have
\begin{align*}
\binom{k-1}{r-1}\binom{k}{r} &= \binom k {r-1} \binom {k-1} r \times
	\frac {k-r+1}{k-r} \\
    &> \binom k {r-1} \binom {k-1} r
\end{align*}
which completes the proof.
\end{proof}

\begin{cor}\label{cor10}
$\T {2k} k > \T {2k} {k+1}$
\end{cor}

\begin{proof}
Using the summation from Proposition \ref{formulas2} again, we have
\[\T {2k} k - \T {2k} {k+1} = \sum_{r=1}^k 2^r 
    \left[\binom{k-1}{r-1}\binom{k+1}r-
	  \binom{k}{r-1}\binom{k}{r}\right]\]  
and this is positive by the preceding lemma.
\end{proof}

\begin{proof}[Proof of Proposition \ref{unimodal2} (Unimodality for
the $2\times n$ table)] {\ }\medskip\par\noindent 
I.  To prove $\T n 0 < \T n 1 < \dots < \T n {\ceil {n/2}}$ for
all $n\ge0$.

The proof is by induction on $n$.  The proposition is easily verified
for small values of $n$; see Table \ref{t2ntable}.  Now, suppose
the proposition is true for all $n<N$, and let $k\le \ceil {N/2}$.

If $N$ is odd and $k=(N+1)/2$ then $N=2k-1$ and by corollary
\ref{cor8}, $\T N k > \T N {k-1}$.

Otherwise it follows that $k\le\ceil{(N+1)/2}$, so we can apply the
inductive hypothesis to all three terms on the right of the expansion
\[\T N k = \T {N-2} {k-1} + \T {N-1} {k-1} + \T {N-1} k\]
to get
\begin{align*}
\T N k	&> \T {N-2} {k-1} + \T {N-1} {k-1} + \T {N-1} k \\
        &= \T {N} {k-1} 
\end{align*}
which proves the proposition for $n=N$; by induction the
proposition holds for all $n$.
\medskip
\par\noindent
II.  To prove $\T n {\ceil {n/2}} > \dots > \T n n$.

Again, we use induction on $n$ and the first few cases are verified
in Table \ref{t2ntable}.  Suppose the proposition is true for all
$n<N$, and let $k\ge \ceil {N/2}$.

If $N$ is even and $k=N/2$ then Corollary \ref{cor10} gives 
$\T N k > \T N {k+1}$.

Otherwise, it follows that $k-1\ge \ceil {(N-1)/2}$ so the inductive
hypothesis applies to all three terms in the expansion
\[\T N k = \T {N-2} {k-1} + \T {N-1} {k-1} + \T {N-1} k\]
and as in the preceding part we conclude $\T N k > \T N {k+1}$.
The proposition follows for all $n$ by induction.
\end{proof}

\begin{rem}
The sequence $1,2,4,8,18,38,88,192,\dots$ formed by taking the
maximum entry from each row of the $\T n k$ table is documented as
Sloane's \href{http://www.research.att.com/~njas/sequences/A110110}{A110110}
\cite{Sl}, in the context of counting Schroder paths.
\end{rem}

\section{Details for the \texorpdfstring{$3\times N$}{3 by N} case}
\def\T#1#2{T(3,#1;#2)}

\begin{proof}[Proof of Proposition \ref{boundary3}]  We want to
show that $\T n k > 0 \iff k\le\floor{(3n+1)/2}$.

1.  For the $(\Leftarrow)$ implication it suffices to show that
there is always an adjacency-free selection of $\floor {(3n+1)/2}$
blocks on a $3\times n$ grid.  Checkerboard selections achieve this
(in two different ways if $n$ is even; in a unique way if $n$ is
odd).

If we write $M(n)$ for the maximum value of $k$ for which $\T n k$
is nonzero, then by inspection, $M(1)=2$ and $M(2)=3$, and this
gives a basis for induction.

Now suppose that $M(n)=\floor{(3n+1)/2}$ for all $n\le N$, where
$N$ is at least 2.

Certainly, the last two columns of any selection without adjacencies
include no more than 3 squares.  So
\begin{align*}
M(N+1)	&\le	M(N-1)+3 \\
	&=	\floor {\frac{3(N-1)+1}{2}+3}\quad\text{by induction} \\
	&=	\floor {\frac{3(N+1)+1}{2}}
\end{align*}
But by the first half of the proof we know that
\[M(N+1)\ge \floor{\frac{3(N+1)+1}{2}}\]
so in fact
\[M(N+1)= \floor{\frac{3(N+1)+1}{2}}\]
and the proposition follows by induction.  This establishes the
extent of the nonzero part of any given row.

We also claimed that $T(n,k)>0\iff n\ge \floor{(2k+1)/3}$; this is
a straightforward corollary of the preceding result which establishes the 
extent of the nonzero part of any given column.
\end{proof}

\def\Tb#1#2{T_b(3,#1;#2)}
\def\Tc#1#2{T_c(3,#1;#2)}
\def\Td#1#2{T_d(3,#1;#2)}

To prove Proposition \ref{poly3}, we introduce a little extra notation:

$\Tb n k$ will stand for the number of adjacency-free selections
on a $3\times n$ grid with only the bottom square of the last column
selected.

$\Tc n k$ will stand for the number of selections with only the
center square of the last column selected.

$\Td n k$ will stand for the number of selections with two squares
in the last column selected.

The subscripts are of course mnemonic for bottom, center, and double.

\begin{lem}
These functions satisfy the following relations:
\begin{align*}
\T n k	&= 2\Tb n k + \Tc n k + \Td n k + \T {n-1} k \tag {R4}\\
\Tb n k	&= \Tb {n-1} {k-1} + \Tc {n-1} {k-1} + \T {n-2} {k-1} \tag {R5}\\
\Tc n k &= 2\Tb {n-1} {k-1} + \Td {n-1} {k-1} + \T {n-2} {k-1} \tag {R6}\\
\Td n k &= \Tc {n-1} {k-2} + \T {n-2} {k-2}\tag {R7} 
\end{align*}
\end{lem}
\begin{proof} Again, these are all self-evident. \end{proof}

Now, to establish Proposition \ref{poly3}, we will prove the following
more detailed result:

\begin{prop}\label{allpoly} For each $k\ge 2$ there exist polynomials
$b_k$, $c_k$, $d_k$ and $p_k$ with the following properties:

\begin{enumerate}
\item $\Tb n k = b_k(n)$, $\Tc n k = c_k(n)$, $\Td n k = d_k(n)$
and $\T n k = p_k(n)$ for all $n\ge k$, and
\item $\deg b_k = \deg c_k = k-1$; $\deg d_k = k-2$; and $\deg p_k=k$.
\end{enumerate}
\end{prop}

\begin{proof}
We can see immediately that 
\begin{align*}
b_2(n)=c_2(n)&=3n-4 \\
\text{and}\quad d_2(n) &=1
\end{align*}
satisfy both conditions, and one may either verify by a direct counting
argument that
\[p_2(n) = \frac{1}{2}(9n^2-13n+6)\]
satisfies the conditions; or observe that, by relation (R4) and the
preceding comments on $b_2$, $c_2$, and $d_2$, the first differences
in the sequence $\left\{\T n 2\right\}_{n=1}^\infty$ are linear;
hence there is a quadratic $p_2$ with $\T n 2 = p_2(n)$ for all
$n\ge 2$.

Now, by way of induction, suppose we have polynomials $b_k$, $c_k$,
$d_k$ and $p_k$ satisfying both conditions of the proposition, for
all $k\le K$.

By (R5)-(R7) and the inductive hypothesis, it follows that there
are polynomials $b_{K+1}$, $c_{K+1}$, and $d_{K+1}$ satisfying both
conditions of the proposition. Then by (R4) we see that the sequence
$\left\{ \T n 2\right\}_{n=K+1}^\infty$  has first differences given
by a polynomial of degree $K$.  So there exists a polynomial $p_{K+1}$
of degree $K+1$ with $T(n,K+1)=p_{K+1}(n)$ for all $n\ge K+1$.

By induction, our polynomials exist as described for all $k$.
\end{proof}

Finally we want to establish the claim made in Proposition \ref{poly3}
about the coefficients on the polynomials $p_k$.  A preparatory
result is helpful; the only virtue of the following lemma is that
it expresses first differences of $p_k$ in terms of $p_{k-1}$,
$p_{k-2}$ and lesser degree polynomials.

\begin{lem}
The polynomials $p_k$, $c_k$, and $d_k$ from Proposition \ref{allpoly}
satisfy the following equation:
\begin{align*}
p_k(n)-p_k(n-1)&=2p_{k-1}(n-1)+p_{k-1}(n-2) \\
&\quad +p_{k-2}(n-2) \\
&\quad +c_{k-2}(n-1)-d_{k-1}(n-1).
\end{align*}
\end{lem}
\begin{proof}
This follows easily from the relations (R4)-(R7).  Note that terms
on the right-hand side are arranged by degree in decreasing order: 
degree $(k-1)$ on the first line; $(k-2)$ on the second; and $(k-3)$
on the third.
\end{proof}

\begin{prop}\label{coeff3}
The polynomials $p_k$ of the preceding proposition have the form
\[p_k(n)=l_k n^k + s_k n^{k-1} + O(n^{k-2})\]
where
\[l_k=\frac {3^k}{k!}\quad\text{and}\quad 
    s_k = \frac{-13 \cdot 3^{k-2}}{2(k-2)!}\]
for all $k\ge 2$.
\end{prop}
\begin{proof}
By induction on $k$.  We have the polynomial $p_2(n)=\frac{1}{2}
(9n^2-13n+6)$ as a basis for induction.

Now, suppose $p_i(n)=l_i n^i + s_i n^{i-1} + O(n^{i-2})$ for all $i<k$.
By the preceding lemma, $p_k(n)$ has first differences
\begin{align*}
p_k(n)-p_k(n-1)&=2\left(l_{k-1}(n-1)^{k-1}+s_{k-1}(n-1)^{n-2}+\dots\right) \\
    &\quad +\left( l_{k-1}(n-2)^{k-1} + s_{k-1}(n-2)^{k-2}+\dots\right) \\
    &\quad +l_{k-2}(n-2)^{k-2} \\
    &\quad +O(n^{k-3}).
\end{align*}
Grouping coefficients, we have
\begin{equation}\label{differences1}
p_k(n)-p_k(n-1) = 3l_{k-1}n^{k-1} 
+ \left(3s_{k-1}-4(k-1)l_{k-1}+l_{k-2}\right)n^{k-2} 
+ O(n^{k-3}).
\end{equation}
At this point it is clear that $p_k$ has the form 
\[ p_k(n) = An^k+Bn^{k-1}+O(n^{k-2}) \]
for \emph{some} coefficients $A$ and $B$.  Taking differences,
\begin{equation}\label{differences2}
p_k(n)-p_k(n-1) = (kA)n^{k-1}
+\left((k-1)B-\binom{k}{2}A\right)n^{k-2}+O(n^{k-3})
\end{equation}
and equating coefficients on $n^{k-1}$ in \eqref{differences1} and \eqref{differences2} we have
\[ kA = 3l_{k-1}, \]
hence $A=l_k$, and we have the correct leading coefficient for
$p_k$.  Then equating coefficients on $n^{k-2}$ gives (after a
little rearrangment)
\begin{align*}
B &= \frac{-13\cdot 3^{k-2}}{2(k-1)(k-3)!}
    -\frac{4\cdot 3^{k-1}}{(k-1)!}
    +\frac{3^{k-2}}{(k-1)!}
    +\frac{3^k}{2(k-1)!} \\
  &= \frac{3^{k-2}\left(-13(k-2)-8\cdot 3+2+9\right)}{2(k-1)!}\\
  &= \frac{-13\cdot3^{k-2}}{2(k-2)!} = s_k,
\end{align*}
so we have the correct second coefficient on $p_k$ as well; induction completes the proof.
\end{proof}

\section{Additional remarks on the \texorpdfstring{$2\times N$}{2 by N} case}
\def\T#1#2{T(2,#1;#2)}

\begin{table}[!htb]
\begin{center}
\begin{tabular}{l|cccccccccccccccc}
$n\backslash k$&1&&2&&3&&4\\
\hline
0&\\
1&2& \boxed{2}& 0\\
2&4& \boxed{2}& 2\\
3&6& \boxed{2}& 8&\boxed{6}& 2\\
4&8&& 18&\boxed{6}& 	12&& 2\\
5&10&& 32&\boxed{6}& 38&\boxed{22}& 16&& \\
6&12&& 50&& 88&\boxed{22}& 66&& \\
7&14&& 72&& 170&\boxed{22}& 192&& \\
\end{tabular}
\caption{Differences near the row maxima in $T(2,n;k)$}\label{maxdiffs}
\end{center}
\end{table}

The following proposition documents a few additional identities
that quickly suggest themselves upon observing the table of $\T n k$
values.  

\begin{prop}
The following identities hold for all $k\ge1$:
\begin{align*}
\T{2k}{k}&=((k+1)/k)\T{2k}{k+1}				\tag{1}\\
\T{2k}{k}-\T{2k}{k+1}&=(1/k)\T{2k}{k+1}			\tag{2}\\
\T{2k}{k}-\T{2k}{k+1}&=\T{2k-1}{k}-\T{2k-1}{k+1}	\tag{3}\\
\end{align*}
\end{prop}
These relate to the difference between a row maximum and one of its
neighbors, as suggested in Table \ref{maxdiffs}.  In particular, (3)
says that $\T{2k}{k}-\T{2k}{k+1}$, the difference between a row
maximum and its right-hand neighbor in an even row, is the same as
the difference of corresponding entries in the previous row (also
between a row maximum and its right-hand neighbor).  Moreover, it
seems that the sequence of differences $2,6,22,\dots$ that we see
emerging is none other than the sequence of large Schroeder numbers
(Sloane's
\href{http://www.research.att.com/~njas/sequences/A006318}{A006318}),
though I have not verified this.

\begin{proof}
All of these are easy to establish.  For the first, we simply apply Proposition \ref{formulas2} and rewrite the summand:
\begin{align*}
\T{2k}{k}
&=\sum_{r=1}^{k} 2^r\binom{k-1}{r-1}\binom{k+1}r \\
&=\sum_{r=1}^{k+1} 2^r\left[\frac{k+1}{k}\binom{k-1}{r-1}\binom{k}{r}\right]\\
&=\frac{k+1}{k}\,\T{2k}{k+1}
\end{align*}
which establishes (1), and of course (2) is immediate from there.
The change in the upper limit of summation in the second line is
harmless, since the summand is zero for $r=k+1$.

For (3), we can begin with the observation that
\begin{align*}
 \binom{k-1}{r-1}\binom{k}{r} - \binom{k}{r-1}\binom{k-1}{r} 
&= \frac{k-r+1}{k}\binom{k}{r-1}\binom{k}{r} 
- \frac{k-r}{k}\binom{k}{r-1}\binom{k}{r} \\
&= \frac{1}{k}\binom{k}{r-1}\binom{k}{r} 
\end{align*}
so applying Proposition \ref{formulas2} to the right-hand side of
(3), we have

\begin{align*}
\T{2k-1}{k}-\T{2k-1}{k+1}
&=\sum_{r=1}^{k+1} \frac{1}{k} 2^r \binom{k}{r-1}\binom{k}{r} \\
&=\frac{1}{k}\,\T{2k}{k+1} \\
&=\T{2k}{k}-\T{2k}{k+1}\quad\text{by (2)}
\end{align*}
and that concludes the proof of (3).  
\end{proof}

\begin{rem}
Table \ref{maxdiffs} also suggests the following conjecture:
\begin{align*}
\T{2k}{k}-\T{2k}{k+1}&=\T{2k+1}{k+1}-\T{2k+1}{k}	\tag{C1}
\end{align*}
That is, that the difference between a row maximum and its right-hand
neighbor in an even row is the same (in absolute value) as the
difference between the corresponding entries in the \emph{next} row
as well.  So far, I have not been able to prove this; a simple
strategy of comparing the sums for the left- and right-hand sides
term-by-term will not work in this case.
\end{rem} 
Finally, the self-similar Sierpinski gasket shape of the mod 2
Pascal's triangle is well-known.  For obvious reasons, the $T(2,n;k)$
table has no interesting mod 2 structure; however, we cannot resist
displaying the following image, which shows 192 rows of the table
colored mod~3 (white for 0, gray for 1, black for 2), and exhibits,
roughly, the structure of a Sierpinski carpet sheared by $45^\circ$.

\begin{figure}[!htb]
\begin{center}
    \includegraphics{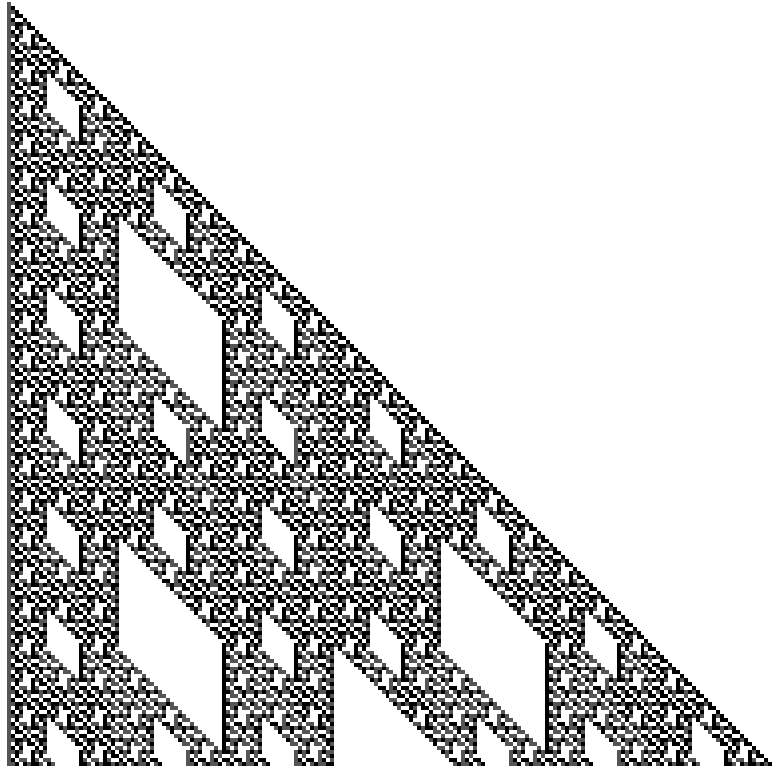}
    \caption{A piece of the $T(2,n;k)$ table colored by remainders mod 3}
\end{center}
\end{figure}

\end{document}